\documentclass{article}
\pdfoutput=1
\usepackage[pdfencoding=auto, psdextra, pagebackref, breaklinks]{hyperref}

\usepackage{L_func}
\usepackage[title]{appendix}

\title{Modules of Zeta Integrals for $\mathrm{GL}(1)$}
\author{Gal Dor\\
Tel-Aviv University}
\date{December 2020}

\begin{document}

\maketitle

\begin{abstract}
    We categorify the Hecke L-functions of $\GL(1)$ by replacing the L-functions with ``modules of zeta integrals''. These modules of zeta integrals are generated by the classical L-function. This approach allows us to categorify questions regarding L-functions, as well as make their construction more canonical by avoiding the GCD procedure usually used to define them.
    
    Let $F$ be a number field, and let $\Chars(F)$ be the space of Hecke characters $\chi\co\AA_F^\times/F^\times\ra\CC^\times$ on $F$. We define a ring $\cS$ of holomorphic functions on $\Chars(F)$ and an $\cS$-module $\cL$ of zeta integrals on $\Chars(F)$. Using a canonical trivialization of the line bundle $\cL$ in some localization $\cS'$ of $\cS$, we show that the module of zeta integrals $\cL$ contains the same information as the Hecke L-function of $\GL(1)$. Here, the L-function is thought of as a single function on $\Chars(F)$.
    
    In the paper \cite{abst_aut_reps_arxiv}, the author has already implicitly used this idea to derive non-trivial applications for the theory of automorphic representations. The goal of this paper make the notion of ``module of zeta integrals'' explicit and rigorous for future applications.
    
    The end result is very closely related to a construction by Connes and Meyer, although seen from an alternative point of view.
    
    This paper is based on a part of the author's thesis \cite{alg_struct_L_funcs}.
\end{abstract}

\tableofcontents

\section{Introduction}

The goal of this paper is to present an alternative but equivalent approach to the classical construction of L-functions for automorphic representations of $\GL(1)$. Philosophically, we are trying to answer the question ``where does an L-function live?'' in the simplest case, Tate's classical construction of Hecke L-functions for $\GL(1)$.

Let $F$ be a number field. There are several ways to construct the L-function for an automorphic representation over $F$. For example, one has Tate's thesis for Hecke L-functions for $\GL(1)$, or the methods of Jacquet--Langlands and Godement--Jacquet for $\GL(n)$. Regardless of the chosen method, one has to choose a \emph{test function} out of some space (e.g., $S(\AA_F)$ for Tate's thesis), to which a zeta integral is assigned. If one chooses the ``right'' test function, whose zeta integral is a GCD for all possible ones, then the resulting zeta integral is called an L-function.

Instead, our approach is to construct a ring $\cS$ and an $\cS$-module $\cL$ of all zeta integrals. It turns out that the module of zeta integrals $\cL$ already contains the same data as the automorphic L-function for $\GL(1)$, and can be thought of as a categorification of the L-function. In this language, the GCD procedure usually used to select a specific L-function out of all zeta integrals becomes a search for a \emph{generator} for the module $\cL$.

If we further specialize to totally real $F$, we can also give an algebro-geometric interpretation to our construction. Let us give a brief overview.

Suppose that $F$ is totally real. Denote by $\Chars(F)$ the space of Hecke characters of $F$. In this paper, we are concerned with the definition of Hecke L-functions as functions on the space $\Chars(F)$. In this setting, the module $\cL$ can be turned into a sheaf of modules on $\Chars(F)$. Sections of this module that are supported on finitely many connected components are zeta integrals, and generators of this module are Hecke L-functions. Similarly, $\cS$ can be turned into a sheaf of rings of functions on the space $\Chars(F)$.

We will use this point of view to algebraically reformulate and categorify several important properties of Hecke L-functions. Put differently, we are proposing that it is beneficial to think of L-functions in more geometric terms, such as modules and sheaves, instead of as specific functions. From this point of view, the existence of an actual \emph{function} which generates our module $\cL$, the L-function itself, is almost coincidental to the theory. Even if the module $\cL$ did not in fact have a generator, the theory could proceed without a problem.

Moreover, it turns out that thinking of L-functions in terms of the module of zeta functions that they generate has unexpected benefits. In the same manner that thinking of vector spaces in basis-free terms allows one to shift the emphasis from the properties of specific elements of vector spaces to the properties of maps between them, one can use this idea to produce new results. In \cite{abst_aut_reps_arxiv}, the author compares two formulas for zeta integrals for $\GL(2)$ that give the same L-function: the Godement--Jacquet construction and the Jacquet--Langlands construction. The fact that they give the same L-function is well-known. But enhancing this into a correspondence between modules of zeta integrals turns out to induce a novel multiplicative structure on the category of $\GL(2)$-modules, with applications to the theory of automorphic representations.

It must be noted that a similar approach has been taken before by Connes (e.g. Section~3.3 of \cite{riemann_F_one}), and Meyer in \cite{zeta_rep}. Connes and Meyer are able to, in the case of $\GL(1)$, canonically construct a virtual representation whose spectrum is the zeroes (minus the poles) of the L-function. This works because for $\GL(1)$, the module $\cL$ is locally free of rank $1$ over $\cS$, which allows one to attach a corresponding divisor $[\cL]-[\cS]$ over the spectrum of $\cS$. The virtual representation constructed by Connes and Meyer is this divisor, and in fact their construction of it goes through constructing our $\cS$ and $\cL$. However, Connes and Meyer put the focus on the ``divisor'' $[\cL]-[\cS]$ instead of the locally free module $\cL$ itself.

\date{\textbf{Acknowledgements: }The author would like to thank his advisor, Joseph Bernstein, for providing the inspiration for this paper, and going through many earlier versions of it. The author would also like to thank both Shachar Carmeli and Yiannis Sakellaridis for their great help improving the quality of this text.}

\subsection{Detailed Summary}

Let us give a more detailed account of this paper's main ideas. Let $\chi\in\Chars(F)$ be some Hecke character $\chi\co\AA^\times/F^\times\ra\CC^\times$. Tate's classical construction of the complete Hecke L-function $\Lambda(\chi,s)$ of $\chi$ works by a GCD procedure. One defines a zeta integral
\begin{equation} \label{eq:tate_zeta_integral}
    \int_{\AA^\times}\Phi(x)\chi(x)\abs{x}^s\dtimes{x}
\end{equation}
for every appropriate \emph{test function} $\Phi\in S(\AA)$. These zeta integrals are all meromorphic functions of $s\in\CC$. Every test function gives a zeta integral; however, some zeta integrals are distinguished. It turns out that the collection of zeta integrals has a greatest common divisor, a meromorphic function $\Lambda(\chi,s)$ such that all zeta integrals are multiples of $\Lambda(\chi,s)$ by an entire function. In this text, we will refer to any such GCD as a \emph{complete L-function} of $\chi$. Note that this terminology is usually used to refer to a specific standard choice of GCD.

The above GCD procedure only defines $\Lambda(\chi,s)$ up to multiplication by an entire invertible function. For some properties of the L-function, this is enough. For example, the zeroes of the L-function are well-defined in this formalism. This formulation also naturally satisfies a functional equation.

However, there are other questions that one would like to ask about the L-function that do not work quite as well with this definition. For example, one is often interested in special values of the L-function. Therefore, in order to speak about specific values of $\Lambda(\chi,s)$, there exists a standard (somewhat ad hoc) choice of GCD in a place-by-place manner in $F$. This expresses $\Lambda(\chi,s)$ as a product of standard L-factors over all places of $F$.

As another example, the growth of the L-function in vertical strips (as $s\ra\sigma\pm i\infty$) is an important problem. However, with the standard choice of GCD, the function $\Lambda(\chi,s)$ decreases rapidly in vertical strips. This is due to the choice of L-factor at the Archimedean places, which decreases so fast that all other behaviour becomes irrelevant. In order to remedy this, the usual approach is to work with the \emph{incomplete L-function} $L(\chi,s)$, which simply drops the L-factors at $\infty$.

In this paper, we will give an alternative approach to the GCD construction. We define a ring $\cS$ of functions over $\Chars(F)$, and a module $\cL$, which we think of as the module of zeta integrals. Our view is that the fundamental object of the theory is this module of zeta integrals. To relate this to the standard view, we will establish a canonical correspondence between generators of $\cL$, i.e. isomorphisms of modules
\[
    t\co\cS\xrightarrow{\sim}\cL,
\]
and a set of functions on $\Chars(F)$ which are complete L-functions under the standard formulation of the GCD procedure (see Construction~\ref{const:gens_are_L_funcs}).

In this manner, claims about Hecke L-functions (growth, factorization under Abelian extensions, etc.) can be refined to explicit claims about the module $\cL$. The main point is that we are able to provide a geometric viewpoint which gives a more geometric flavor to the GCD definition of L-functions, i.e. exhibits them as generators of a module.

A less formal, but perhaps more descriptive, explanation is as follows. We introduce the ring $\cS$, its module $\cL$, and a canonical isomorphism of $\cL$ with an extension $\cS'$ of $\cS$ after base change:
\[
    \cS'\otimes_\cS \cL\xrightarrow{\sim}\cS'.
\]
Geometrically, the extension $\cS\subseteq\cS'$ corresponds to restricting to a specific right-half-plane. This turns $\cL$ into a kind of ``divisor'' on the space $\Chars(F)$, supported on the corresponding left-half-plane. This ``divisor'' turns out to be principal (in an informal sense), and L-functions are precisely those functions which generate it (see Remark~\ref{remark:L_is_divisor} for more details). The virtual representation constructed by Meyer in \cite{zeta_rep} is the sky-scraper sheaf at this divisor, $[\cL]-[\cS]$.

On top of providing an interesting perspective, this will be useful in several immediate ways. First of all, the proposed reformulation will enable us to refine some properties of L-functions as algebraic statements about the module $\cL$. We will include examples such as the functional equation, and more precise statements about decomposition properties of L-functions under Abelian extensions. Specifically, it turns out that the fact that L-functions of Abelian extensions decompose into a product of L-functions of characters can be categorified into an explicit canonical decomposition of the corresponding spaces of zeta integrals. See Remark~\ref{remark:cubic_L_decomposes} for an interesting application.

Moreover, the set of functions on $\Chars(F)$ given by the algebraic construction is somewhat smaller than that given by the GCD construction. This means that for some statements about the growth of the L-function as $s\ra\sigma\pm i\infty$, no non-canonical choice of representative is necessary.

Let us elaborate. It turns that out of the complete L-functions given by the GCD construction, the subset that correspond to generators of $\cL$ always have an additional property: they have moderate growth and decay properties in vertical strips. Hence, it makes sense to talk about their growth without introducing the incomplete L-function, or even choosing a standard representative. In other words, the geometric point of view specifies not only the zeroes of the L-function, but also its growth properties. Technically, this happens because the new formalism rejects the standard L-factors at the Archimedean places, which decay much too quickly in vertical strips to generate $\cL$. Instead, different Archimedean local L-factors are required, which ``fixes'' their growth. See Remark~\ref{remark:growth_of_L_in_vertical_strips} and Appendix~\ref{app:non_canonical_triv} for further discussion of this issue.

Another kind of result that can be achieved using the geometric point of view is the trace formula given in \cite{zeta_rep}, where Meyer essentially derives the explicit formula for Hecke L-functions from the geometric construction without any reference to the underlying L-function. However, we will not pursue this direction here.

The organization of this paper is as follows. Section~\ref{sect:dixmier_malliavin} contains a brief reminder about the Dixmier-Malliavin theorem, following \cite{dixmier_malliavin_for_born_arxiv}. Section~\ref{sect:module_S} introduces the ring $\cS$ of functions on $\Chars(F)$, and studies its properties. Section~\ref{sect:module_L} introduces the module $\cL$. Section~\ref{sect:canonical_triv} establishes the correspondence between generators of $\cL$ and L-functions. Section~\ref{sect:functional_equation} re-introduces the functional equation in the new language. Section~\ref{sect:decomposition_under_ext} studies the decomposition properties of L-functions under Abelian extensions. Finally, Appendix~\ref{app:non_canonical_triv} shows that a generator for the module $\cL$ actually exists, and discusses the various local L-factors given by the new formalism.

\begin{remark}
    While $\GL(1)$ is a nice test-case, the theory of L-functions and automorphic representations often focuses on higher rank reductive groups, such as $\GL(n)$. While the contents of this text can be generalized in various ways to (say) $\GL(2)$, the non-commutativity of the group adds a severe technical complication that we will not handle here.
    
    It is the author's belief that the symmetric monoidal structure constructed in \cite{abst_aut_reps_arxiv}, which seems to help $\GL(2)$ manifest some ``commutative-like'' phenomena, would be useful in such an endeavor. See Remarks~3.1 and~4.1 of \cite{abst_aut_reps_arxiv} for details on these phenomena.
\end{remark}

\begin{remark}
    When working over number fields, a major technical tool that we will use is the notion of \emph{bornological vector space}. This notion greatly resembles that of the more commonly used \emph{topological vector space}, but is technically simpler and more suitable for studying the representation theory of locally compact groups. For example, the ring $\cS$ and module $\cL$ can be given the structure of a bornological ring and bornological module, respectively.
    
    However, bornologies are not the focus of this paper. Rather, we consider bornological vector spaces because they admit a strong version of the Dixmier-Malliavin theorem, which we want to use. Specifically, we will need the variant given in \cite{dixmier_malliavin_for_born_arxiv}.
    
    Moreover, this notion is not actually needed when $F$ is a function field, as the Dixmier-Malliavin theorem becomes trivial in this case. Readers can safely take $F$ to be a function field, and consequently ignore all bornological structures appearing in this text.
    
    We will give a brief reminder on the Dixmier-Malliavin theorem in Section~\ref{sect:dixmier_malliavin}.
\end{remark}

\section{Reminder on the Dixmier-Malliavin Theorem} \label{sect:dixmier_malliavin}

This section is a brief reminder (closely following \cite{dixmier_malliavin_for_born_arxiv}) about bornological vector spaces and their relation with the Dixmier-Malliavin theorem, which are used in the main body of the text. This section should not be considered original contribution. Readers who are only interested in the case of L-functions over function fields can safely skip this section and ignore all mentions of bornological structures.

Bornological vector spaces will be used in much of this text as a technically favorable alternative for topological vector spaces. In many applications, the two notions are very similar. However, while they are close enough that many theorems can be successfully stated both in terms of bornologies and in terms of topologies, there are still cases where one language is preferable to the other. For example, there are some theorems (especially in the representation theory of locally compact groups) that have many technical requirements when stated in the traditional language of topological vector spaces. However, these extra technical assumptions disappear when the theorem is stated in the language of bornological vector spaces instead.

The main way we will use the notion of a bornological vector space in this text is because they support a stronger variant of the Dixmier-Malliavin theorem. More concretely, we will use bornological structures to prove that certain rings satisfy a property called quasi-unitality.

Specifically, we say that a (non-unital) ring $R$ is \emph{quasi-unital} if the product map:
\[
    R\otimes_R R\ra R
\]
is an isomorphism, where the relative tensor product $R\otimes_R R$ is the quotient of $R\otimes R$ by all expressions of the form
\[
    (ab)\otimes c-a\otimes(bc).
\]
Similarly, if $R$ is a quasi-unital ring, and $M$ is an $R$-module, then we say that $M$ is \emph{smooth} if the action map:
\[
    R\otimes_R M\ra M
\]
is an isomorphism. Once again, we take $R\otimes_R M$ to be the relative tensor product.

Let $G$ be an algebraic group, and let $F$ be a number field. Let $G(\AA)=G(\AA_F)$ denote the adelic points of $G$ over $F$. Let $C_c^\infty(G(\AA))$ be the (non-unital) ring of smooth and compactly supported functions on $G$, equipped with the convolution product. The ring $C_c^\infty(G(\AA))$ is naturally a bornological non-unital ring. The main result we need from \cite{dixmier_malliavin_for_born_arxiv} is the following variant of the Dixmier-Malliavin theorem:
\begin{theorem}[Theorem~5.1 of \cite{dixmier_malliavin_for_born_arxiv}] \label{thm:adelic_garding_is_smooth}
    The following hold:
    \begin{enumerate}
        \item The ring $C_c^\infty(G(\AA))$ is quasi-unital.
        \item Let $V$ be a complete bornological vector space, equipped with a smooth action of $G(\AA)$ on $V$. Then $V$ is smooth as a $C_c^\infty(G(\AA))$-module.
    \end{enumerate}
\end{theorem}

\begin{remark}
    Essentially, what Theorem~\ref{thm:adelic_garding_is_smooth} means is that given some boundedness and analytic conditions on a representation of $G(\AA)$, then that representation is a smooth $C_c^\infty(G(\AA))$-module. Note that while the prerequisites for the theorem are analytic in nature (the existence of a complete bornology on $V$, such that the action of $G(\AA)$ is smooth with respect to it), the consequence of the theorem -- smoothness as a module of a quasi-unital ring -- is purely algebraic:
    \begin{equation} \label{eq:smoothness_of_module}
        C_c^\infty(G(\AA))\otimes_{C_c^\infty(G(\AA))}V\xrightarrow{\sim} V.
    \end{equation}
    That is, Equation~\eqref{eq:smoothness_of_module} holds as an isomorphism of vector spaces, and requires no notion of bornology or completeness to state.
    
    Our uses of bornologies in this paper will almost exclusively be as a way to establish this algebraic property. That is, the main focus of this paper is algebraic, rather than analytic.
\end{remark}

\begin{remark}
    The proof of Theorem~5.1 of \cite{dixmier_malliavin_for_born_arxiv} is written for the case of a Lie group $G$. The generalization to the adelic case follows Remark~5.3 of \cite{dixmier_malliavin_for_born_arxiv}.
\end{remark}

\begin{remark}
    This theorem becomes trivially true in the case that $F$ is a function field, and $C_c^\infty(G(\AA))$ is considered the ring of smooth and compactly supported functions on $G$ in the usual sense of being locally constant. In this case, no assumptions about bornologies are necessary.
\end{remark}

For more details, we direct the reader to \cite{dixmier_malliavin_for_born_arxiv}. See also \cite{borno_quasi_unital_algs2}, which discusses bornological structures in representation theory specifically, and \cite{borno_vs_topo_analysis}, which deals more generally with bornologies and functional analysis.

\section{The Space of Automorphic Functions \texorpdfstring{$\cS$}{S}} \label{sect:module_S}

Let $F$ be a number field (not necessarily totally real), and let $\Chars(F)$ be the space of Hecke characters $\chi\co\AA_F^\times/F^\times\ra\CC^\times$ with its complex analytic topology. That is, $\Chars(F)$ is a countable disjoint union of copies of $\CC$, each parametrized by $\abs{\cdot}^s\chi$ with $s\in\CC$, for some unitary Hecke character $\chi$. We will say that the component of all Hecke characters of the form $\abs{\cdot}^s\chi$ is the component \emph{corresponding to $\chi$}.

\begin{remark} \label{remark:vertical_strips}
    We will often be speaking about functions on $\Chars(F)$ that are rapidly decreasing in vertical strips. This should be taken to mean that $\abs{f(\sigma+it)}\ra 0$ as $t\ra\infty$ faster than any polynomial in vertical strips $a\leq \sigma\leq b$, on every copy of $\CC$ separately.
\end{remark}

On the space $\Chars(F)$, there is a ring of (Mellin transforms of) Bruhat-Schwartz functions, $\cS_F$. Our main construction is a module $\cL_F$ over $\cS_F$, whose generators will correspond to L-functions.

In this section, we will construct the space $\cS_F$ itself, and establish some of its basic properties. Instead of directly defining this, let us define its Fourier-Mellin transform, which might be more easily accessible:
\begin{definition}
    Let $\Schw_F=S(\AA^\times)$ denote the (non-unital) ring of Bruhat-Schwartz functions on $\AA^\times$. Specifically, we set
    \[
        S(\AA^\times)=S(F_\infty^\times)\otimes{\bigotimes_{p}}'S(F_p^\times),
    \]
    where
    \begin{itemize}
        \item The symbol $\bigotimes'$ denotes the restricted tensor product over the finite (i.e., non-Archimedean) places of $F$, with respect to the characteristic function $\one_{\O_p^\times}$.
        \item For a non-Archimedean place $p$, the space $S(F_p^\times)$ is the space of smooth and compactly supported functions on $F_p^\times$.
        \item For the Archimedean places, the space $S(F_\infty^\times)$ is the space of smooth functions $f$ on $F_\infty^\times=\prod_{v|\infty}F_v^\times$, satisfying that
        \[
            \abs{\chi(y)\cdot Df(y)}
        \]
        is bounded for all multiplicative characters $\chi\co F_\infty^\times\ra\CC^\times$ and differential operators $D$ in the universal enveloping algebra of the Lie algebra of the real group $F_\infty^\times$.
    \end{itemize}
    We give $\Schw_F=S(\AA^\times)$ a ring structure via the convolution product with respect to the standard Haar measure on $\AA^\times$ (normalized as in, e.g., page~46 of \cite{aut_reps_book_I}).
    
    When the number field $F$ is clear from context, we will abuse notation and denote $\Schw=\Schw_F$.
\end{definition} 

\begin{remark}
    The ring $\Schw$ naturally acquires a bornology as follows.
    
    We give each of the spaces $S(F_p^\times)$ a bornology consisting of the bounded subsets of its finite-dimensional linear subspaces. Likewise, we give $S(F_\infty^\times)$ a bornology consisting of the subsets where for each $\chi$ and $D$, the expression $\abs{\chi(y)\cdot Df(y)}$ above is bounded uniformly for $f$ and $y$.

    The bornivorous topology associated to this bornology on $\Schw$ is the usual topology of Bruhat-Schwartz functions on $\AA^\times$.
\end{remark}

\begin{remark} \label{remark:schwartz_at_infty_decay_exp}
    An alternative way to view $S(F_\infty^\times)$ is as follows. By applying the product of the logarithm maps over the Archimedean places $v|\infty$, we can identify $F_\infty^\times$ with a disjoint union of a finite number of copies of $\RR^{r_1+r_2}\times\left(\RR/\ZZ\right)^{r_2}$. Under this identification, the space $S(F_\infty^\times)$ corresponds to smooth functions, all of whose derivatives are decreasing faster than any exponential in the \emph{logarithmic} coordinates $\RR^{r_1+r_2}$.
\end{remark}

\begin{remark}
    The reader should note that we are actually working with \emph{measures}, rather than \emph{functions}, on the space $\AA^\times$ (since we are taking their integrals and convolutions). However, for the sake of simplicity, we are relying on the standard Haar measure $\dtimes{g}$ (see, e.g., page~46 of \cite{aut_reps_book_I} for an explicit description of this standard normalization) to abuse notation and speak of functions regardless. This is done in order to simplify the notation and exposition, while hopefully not confusing the reader too much.
\end{remark}

\begin{remark}
    The bornological space $\Schw=S(\AA^\times)$ is complete, and the action of $\AA^\times$ on it is smooth. Thus, by Theorem~\ref{thm:adelic_garding_is_smooth} and Claim~3.20 of \cite{dixmier_malliavin_for_born_arxiv}, the ring $S(\AA^\times)$ is quasi-unital. Note that this is a purely-algebraic property.
\end{remark}

\begin{definition}
    Let $\cS_F=\Schw_F/F^\times$ be the space of co-invariants of $\Schw_F=S(\AA^\times)$ by the action of $F^\times$ via multiplicative shifts.
    
    When the number field $F$ is clear from context, we will abuse notation and denote $\cS=\cS_F$.
\end{definition}

\begin{remark}
    We note that the map
    \[
        f(g)\mapsto\sum_{q\in F^\times}f(qg)
    \]
    defines an isomorphism of $\cS=\Schw/{F^\times}$ with what is sometimes known as the space of Bruhat-Schwartz functions on $\AA^\times/F^\times$. Thus, we will sometimes refer to $\cS$ as the \emph{space of automorphic functions}. 
\end{remark}

\begin{remark} \label{remark:coinv_are_closed}
    It is possible, via standard techniques, to construct a bounded section
    \[
        \cS\ra\Schw
    \]
    for the canonical projection. In particular, $\cS$ is a complete bornological space.
\end{remark}

\begin{remark}
    Because $\Schw$ is commutative, the space $\cS$ is a ring. Once again, Theorem~\ref{thm:adelic_garding_is_smooth} and Claim~3.20 of \cite{dixmier_malliavin_for_born_arxiv}, along with the fact that $\cS$ is complete by Remark~\ref{remark:coinv_are_closed}, let us conclude that $\cS$ is quasi-unital.
\end{remark}

\begin{remark} \label{remark:paley_wiener}
    When $F$ is totally real, one can use the Paley-Wiener theorem to give an alternative description for the space $\cS$. It is isomorphic via the Mellin transform to the space of functions on $\Chars(F)$ which are supported only on finitely many copies of $\CC$, and on each one they are entire functions $f(s)$ which are rapidly decreasing in vertical strips.
    
    In this view, the bornology of $\cS=\bigoplus\cS_{\chi_0}$ is the direct sum bornology induced from the subspaces $\cS_{\chi_0}$ of functions $f\in\cS$ whose Mellin transform is supported on just one copy of $\CC$ (corresponding to $\chi_0\in\Chars(F)$). The bornology on the subspaces $\cS_{\chi_0}$ themselves is the von-Neumann bornology for the topology generated by the semi-norms
    \[
        \norm{\hat{f}(s)}_{\sigma,n}=\sup_{t\in\RR}\,(1+\abs{t}^n)\abs{\hat{f}(\sigma+it)}
    \]
    for $\sigma\in\RR$ and $n\geq 0$, with
    \[
        \hat{f}(s)=\int_{\AA^\times/F^\times}f(y)\chi_0(y)\abs{y}^s\dtimes{y}
    \]
    the Mellin transform of $f$.
\end{remark}

\begin{remark}
    One can give a similar, but more complicated, description when $F$ is not totally real. See also Remark~\ref{remark:paley_wiener_at_C} for a similar issue.
\end{remark}

\begin{remark} \label{remark:schwatz_is_co_sheaf}
    A convenient way to assign geometric intuition to the space $\cS$ is to think of it as a \emph{co-sheaf} on $\Chars(F)$. This should express the fact that the Mellin transforms of elements of $\cS$ are supported on a finite number of copies of $\CC$.
    
    To be more explicit, suppose that $F$ is totally real as above. We define a co-sheaf on $\Chars(F)$ as follows. To each connected open set $U$ of $\Chars(F)$ (which necessarily lies inside a single copy of $\CC$, corresponding to $\chi_0\in\Chars(F)$), the co-sheaf assigns the subspace $\cS_{\chi_0}\subseteq\cS$ of functions supported on that copy. For arbitrary open sets $U\subseteq\Chars(F)$, we let the value of the co-sheaf be the direct sum of the values of the co-sheaf on the connected components of $U$. This turns $\cS$ into the global co-sections of a locally constant co-sheaf on $\Chars(F)$.
\end{remark}

\section{The Module of Zeta Integrals \texorpdfstring{$\cL$}{L}} \label{sect:module_L}

We are now ready for the main construction of this paper. We describe a module (of bornological vector spaces) $\cL$ over the ring $\cS$. The elements of this module will correspond to zeta integrals, and its generators $t\co\cS\xrightarrow{\sim}\cL$ will correspond to L-functions. We will show this correspondence explicitly in Section~\ref{sect:canonical_triv}.

\begin{definition}
    Let $S(\AA)$ be the $\Schw=S(\AA^\times)$-module of Bruhat-Schwartz functions on $\AA$. Specifically, we set
    \[
        S(\AA)=S(F_\infty)\otimes{\bigotimes_{p}}'S(F_p),
    \]
    where
    \begin{itemize}
        \item The symbol $\otimes'$ denotes the restricted tensor product over the finite (i.e., non-Archimedean) places of $F$, with respect to the characteristic function $\one_{\O_p}$.
        \item For a non-Archimedean place $p$, the space $S(F_p)$ is the space of smooth and compactly supported functions on $F_p$.
        \item For the Archimedean places, the space $S(F_\infty)$ is the space of Schwartz functions $f$ on $F_\infty=\prod_{v|\infty}F_v$.
    \end{itemize}
\end{definition}

\begin{remark}
    Note that $S(\AA)$ is a bornological $\Schw$-module.
\end{remark}

\begin{remark}
    Let us take a different point of view on $S(F_\infty)$. Under the identification of $F_\infty^\times$ as a finite disjoint union of copies of $\RR^{r_1+r_2}\times\left(\RR/\ZZ\right)^{r_2}$ (as in Remark~\ref{remark:schwartz_at_infty_decay_exp}), we can look at the restriction $f|_{F_\infty^\times}$ of a function $f\in S(F_\infty)$. This restriction approaches a limit as any subset of the coordinates approaches $-\infty$, and decays faster than any exponent as any of the coordinates approaches $\infty$.
\end{remark}

\begin{definition}
    We let $\cL_F$ denote the vector space of $F^\times$ co-invariants of $S(\AA)$, where $F^\times$ acts via multiplication on $\AA$.
    
    When the number field $F$ is clear from context, we will omit it from the notation and denote $\cL=\cL_F$.
\end{definition}

\begin{remark}
    It is immediate to see that $\cL$ is a bornological $\cS$-module.
\end{remark}

\begin{remark}
    Just like in Remark~\ref{remark:schwatz_is_co_sheaf}, one can gain some geometric intuition about $\cL$ by thinking of it as a co-sheaf.
    
    Suppose that $F$ is totally real. Then one can use the co-sheaf structure on $\cS$ defined in Remark~\ref{remark:schwatz_is_co_sheaf}, combined with the module structure of $\cL$, to turn $\cL$ into a co-sheaf as well.
    
    To be as explicit as possible, we construct a co-sheaf on $\Chars(F)$ as follows. To each connected open set $U$ of $\Chars(F)$, which necessarily lies inside a single copy of $\CC$ corresponding to some $\chi_0\in\Chars(F)$, the co-sheaf assigns the subspace $\cL_{\chi_0}=\cS_{\chi_0}\cdot \cL\subseteq\cL$. For arbitrary open sets $U\subseteq\Chars(F)$, we let the value of the co-sheaf be the direct sum of its values on the connected components of $U$. This turns $\cL$ into the global co-sections of a co-sheaf.
\end{remark}

General $\cS$-modules can be pretty badly behaved. However, the specific $\cS$-module $\cL$ is as nice as it can be -- in fact, it is secretly isomorphic to $\cS$, albeit in a non-canonical fashion. This fact will be proven in Appendix~\ref{app:non_canonical_triv}.

More specifically, combining Remark~\ref{remark:coinv_are_closed} with Claim~\ref{claim:A_is_A_times}, we see that the bornological vector space $\cL$ is a complete, smooth $\cS$-module. In particular, we have the algebraic property that the action map
\[
    \cS\otimes_\cS\cL\xrightarrow{\sim}\cL
\]
is an isomorphism.

\section{Correspondence with L-Functions} \label{sect:canonical_triv}

One way to think of $\cL$ is as a module of zeta integrals. Indeed, given a function $f\in \cL=S(\AA)_{/F^\times}$, by restricting it to $\AA^\times/F^\times$ and applying the Mellin transform, we precisely get a zeta integral. This is the content of Construction~\ref{const:L_into_S_prime} below, which will let us relate $\cL$ to the classical notion of an L-function, via Construction~\ref{const:gens_are_L_funcs}.

Let us give some more details. In order to have a well-defined notion of zeta integral, we will need to extend the ring $\cS$ to a sufficiently large ``ring of periods'' $\cS'$ to contain all necessary integrals. We will think of the extension $\cS\subseteq\cS'$ as the space of holomorphic functions on some ``right-half-plane'' in $\Chars(F)$. Along with $\cS'$, we will obtain a canonical trivialization
\begin{equation} \label{eq:informal_triv}
    \cS'\otimes_\cS\cL\xrightarrow{\sim}\cS',
\end{equation}
sending every element of $\cL$ (which is, essentially, a test function coming from $S(\AA)$) to its corresponding zeta integral. This is the algebraic structure which captures, in our setting, the notion of a zeta integral.

\begin{remark} \label{remark:L_is_divisor}
    There is another point of view on this issue as well, with a more geometric flavor. Consider the following informal analogy. Let $X$ be a scheme, with sheaf of functions $\O_X$. In our analogy, these correspond to $\Chars(F)$ and $\cS$, respectively. Suppose that we are given some localization $\O_X\subset\O_{X'}$ (which is our $\cS'$), allowing poles at certain places in $X$.
    
    A line bundle on $X$ is a locally free $\O_X$-module of rank one. The data of a line bundle, along with its trivialization after base change to $\O_{X'}$, is the data of a \emph{divisor} supported at $X-X'$.
    
    Therefore, geometrically, the data of the module $\cL$ along with its trivialization~\eqref{eq:informal_triv} can be thought of as a kind of ``divisor'' on the space $\Chars(F)$. The fact that $\cL$ is non-canonically isomorphic to $\cS$ is then a statement that this divisor is principal, and a function giving this principal divisor is an L-function.
\end{remark}

Since our only real requirement from $\cS'$ is that it is sufficiently large, its construction is relatively ad hoc. We will outline one possible choice for $\cS'$ in Subsection~\ref{subsect:module_S_prime}. In Subsection~\ref{subsect:L_gen_is_L_func}, we will construct the canonical trivialization~\eqref{eq:informal_triv} after base-change to $\cS'$.

\subsection{The Extension \texorpdfstring{$\cS'$}{S'}} \label{subsect:module_S_prime}

In this subsection, we will construct the extension $\cS\subseteq\cS'$, such that $\cS$ and $\cL$ become canonically isomorphic after base change to $\cS'$.

This data will be used to give a correspondence between isomorphisms $t\co\cS\xrightarrow{\sim}\cL$ of $\cS$-modules (which we think of as \emph{generators} for $\cL$) and certain L-functions $L_t$ on the space $\Chars(F)$. This correspondence justifies thinking of $\cL$ as containing the data of an L-function. See Construction~\ref{const:gens_are_L_funcs}.

Our immediate goal is to extend $\cS$ to some minimal extent that also allows $\cL$ to fit into it. This will be our choice of $\cS'$. Intuitively, we want $\cS'$ to somehow correspond to holomorphic functions on some right half-plane in $\Chars(F)$ where L-functions are absolutely convergent.

Let us begin by explicitly constructing $\cS'$. Let $\norm{g}_{\AA}$ be the height function on $\AA$ given by
\begin{equation*}
    \norm{g}_\AA=\prod_v\max\{\norm{g_v}_v,1\},
\end{equation*}
which we restrict to $\AA^\times$. Similarly, we define a height function on $\AA^\times/F^\times$ by choosing the lowest lift:
\[
    \norm{g}_{\AA/F^\times}=\inf_{\stackrel{g'\in\AA^\times}{g\equiv g'\pmod{F^\times}}}\norm{g'}_\AA.
\]

\begin{construction} \label{const:S_prime}
    The ring extension $\cS'$ of $\cS$ is given by the space of smooth functions $f$ on $\AA^\times/F^\times$ such that:
    \begin{itemize}
        \item The stabilizer of $f$ is open in $\AA_\text{fin}^\times$.
        \item There is a bound:
        \[
            \int_{\AA^\times/F^\times}\abs{g}^{1+\varepsilon}\norm{g}_{\AA/F^\times}^\sigma\cdot\abs{Df(g)}\dtimes{g}<\infty
        \]
        for all $\sigma<\infty$, $\varepsilon>0$ and $D$ in the universal enveloping algebra of $F_\infty^\times$.
    \end{itemize}
\end{construction}
\begin{remark}
    The convolution product turns $\cS'$ into a ring by observing that
    \[
        \norm{gg'}_{\AA/F^\times}\leq\norm{g}_{\AA/F^\times}\norm{g'}_{\AA/F^\times}.
    \]
\end{remark}

\begin{remark}
    The ring $\cS'$ has a natural bornological structure, according to which it is complete. By Theorem~\ref{thm:adelic_garding_is_smooth}, it follows that $\cS'$ is smooth over $\cS$.
\end{remark}

\begin{remark} \label{remark:S_prime_is_localization}
    The complete bornological ring $\cS'$ induces a localization functor
    \[
        V\mapsto\cS'\widehat{\otimes}_\cS V
    \]
    on the category of complete bornological smooth $\cS$-modules. Here, $\widehat{\otimes}_\cS$ is the completion of the relative tensor product.
    
    Indeed, we construct a natural morphism $V\ra\cS'\widehat{\otimes}_\cS V$ via
    \[\xymatrix{
         V & \cS\otimes_\cS V \ar[l]_-\sim \ar[r] & \cS'\widehat{\otimes}_\cS V.
     }\]
    Thus, it remains to verify that the completion of the multiplication map
    \[
        \cS'\widehat{\otimes}_\cS\cS'\ra\cS'
    \]
    is an isomorphism. This follows because $\cS'{\otimes}_\cS\cS'$ and $\cS'{\otimes}_{\cS'}\cS'$ share a dense subset $\cS$, with the same induced bornology.
\end{remark}

\begin{remark} \label{remark:S_prime_paley_wiener_in_lines}
    Suppose that $F$ is totally real. Then we can try giving a description of the image of $\cS'$ under the Mellin transform. This should give some geometric intuition about the ring $\cS'$.
    
    Indeed, using Remark~\ref{remark:paley_wiener}, we see that
    \[
        \cS'=\bigoplus\cS'_{\chi_0},
    \]
    where each $\cS'_{\chi_0}$ consists of functions on the right half plane $\{\abs{\cdot}^s\chi_0\suchthat\Re{s}>1\}$ of $\Chars(F)$ which are analytic in the right half-plane and rapidly decreasing in vertical strips there (in the sense of Remark~\ref{remark:vertical_strips}). Recall that we choose unitary representatives $\chi_0$ out of each connected component of $\Chars(F)$.
\end{remark}

\subsection{Canonical Trivialization of the Module of Zeta Integrals} \label{subsect:L_gen_is_L_func}

Our goal for this subsection is to construct the isomorphism $\cS'\otimes_\cS\cL\xrightarrow{\sim}\cS'$, which will be induced from a map $\cL\ra\cS'$. This will let us provide the desired correspondence $t\mapsto L_t$ between generators $t\co\cS\xrightarrow{\sim}\cL$ and L-functions $L_t$.

\begin{construction} \label{const:L_into_S_prime}
    There is a canonical morphism of bornological $\cS$-modules
    \[
        \cL\ra\cS'
    \]
    given by:
    \[
        f(g)\mapsto \sum_{q\in F^\times}f(qg).
    \]
\end{construction}

\begin{remark}
    One can similarly define a map
    \[
        \cS\ra\cS',
    \]
    using the same formula.
    
    It is easy to check that the map $\cS\ra\cS'$ is injective, by using standard techniques.
\end{remark}

\begin{remark}
    We will see later (Corollary~\ref{cor:L_inj_into_S_prime}) that the morphism $\cL\ra\cS'$ is also injective.
\end{remark}

\begin{remark}
    The morphism $\cL\ra\cS'$ of Construction~\ref{const:L_into_S_prime} already encodes all zeta integrals. To see this, suppose that we had some test function $\Psi\in S(\AA)$. Applying the map $\cL\ra\cS'$ along with the Mellin transform, we obtain the corresponding zeta integral via Equation~\eqref{eq:tate_zeta_integral}.
\end{remark}

Our main result for this section is that the map $\cL\ra\cS'$ of Construction~\ref{const:L_into_S_prime} presents $\cS'$ as the localization of $\cL$ under the functor $\cS'\hat{\otimes}_\cS-$ of Remark~\ref{remark:S_prime_is_localization}. In other words, $\cS'$ defines a localization of the category of complete smooth $\cS$-modules under which $\cS$ and $\cL$ coincide. This will justify thinking of $\cL$ together with the data of the map $\cL\ra\cS'$ as defining a ``divisor'' on $\Chars(F)$.
\begin{theorem} \label{thm:S_prime_L_isom}
    The morphism $\cL\ra\cS'$ of Construction~\ref{const:L_into_S_prime} induces an isomorphism:
    \[
        \cS'\otimes_\cS\cL\xrightarrow{\sim}\cS'.
    \]
\end{theorem}

\begin{remark}
    In more classical terms, the content of Theorem~\ref{thm:S_prime_L_isom} is that the L-functions we are constructing have a zero-free right-half-plane.
\end{remark}

The proof of Theorem~\ref{thm:S_prime_L_isom} is fairly heavy, and will be postponed to the end of this section.

At this point, we have all of the constructions necessary for our notion of L-function. We claim that the module $\cL$, together with the canonical embedding $\cL\ra\cS'$, encodes the classical notion of L-function. We formalize this statement as the following construction:
\begin{construction} \label{const:gens_are_L_funcs}
    Suppose that we are given some (non-canonical) isomorphism
    \[
        t\co\cS\xrightarrow{\sim}\cL.
    \]
    Applying the functor $\cS'\otimes_\cS-$, we get a morphism of modules
    \[
        t\co\cS'\xrightarrow{\sim}\cS'.
    \]
    Recall that this is sometimes called an element of the \emph{roughening} of $\cS'$. After composing with the Mellin transform from both sides, as in Remark~\ref{remark:S_prime_paley_wiener_in_lines}, the map $t$ acts as multiplication by a function $L_t$ on some right-half-plane of the space $\Chars(F)$. We refer to $L_t$ as the \emph{L-function corresponding to $t$}.
\end{construction}

\begin{remark}
    Construction~\ref{const:gens_are_L_funcs} is not vacuous. Specifically, Theorem~\ref{thm:L_is_triv} guarantees the existence of a generator $t\co\cS\xrightarrow{\sim}\cL$ as above.
\end{remark}

\begin{corollary} \label{cor:L_inj_into_S_prime}
    The morphism $\cL\ra\cS'$ is injective.
\end{corollary}
\begin{proof}
    Pick an isomorphism $t\co\cS\xrightarrow{\sim}\cL$, as in Theorem~\ref{thm:L_is_triv}. Then we get a commutative diagram:
    \[\xymatrix{
        \cS \ar[d]^t & \Schw\otimes_\Schw\cS \ar[l]_-\sim \ar[r] & \cS'\otimes_\cS\cS \ar[d]^{\cS'\otimes_\cS\,\displaystyle t} \ar[r]^-\sim & \cS' \ar@{-->}[d]^{L_t} \\
        \cL & \Schw\otimes_\Schw\cL \ar[l]_-\sim \ar[r] & \cS'\otimes_\cS\cL \ar[r]^-\sim & \cS'.
    }\]
    By Theorem~\ref{thm:S_prime_L_isom}, we can extend $\cS'\otimes_\cS t$ to an isomorphism $\xymatrix@1{\cS' \ar@{-->}[r]^{L_t} & \cS'}$. Since the top row of the diagram is an injective morphism, so is the bottom row.
\end{proof}

\begin{remark}
    Let $t\co\cS\xrightarrow{\sim}\cL$ be some generator for $\cL$. Then $L_t$ is an automorphism of $\cS'$. In particular, it has no zeroes with $\Re{s}>1$. Moreover, it satisfies some moderate growth condition in vertical strips.
    
    We note that this L-function is well defined up to composing $t$ with an automorphism of $\cS$ as a module over itself. I.e., up to the roughening of $\cS$. This makes the L-function $L_t$ well-defined up to multiplication by an entire function such that both it and its inverse have no zeroes, and satisfy a similar moderate growth condition in vertical strips.
\end{remark}

\begin{remark} \label{remark:growth_of_L_in_vertical_strips}
    One should note that the above is slightly stronger than the corresponding classical claim. That is, when the L-function $\Lambda(\chi,s)$ is defined via the GCD procedure, it is only well-defined up to multiplication by an entire function with no zeroes. I.e., the moderate growth condition is absent.
    
    The author finds this interesting. Usually, when one considers growth conditions on, say, the Riemann zeta function, one takes the non-completed zeta function $\zeta(s)$. The reason for this is that the completed zeta function
    \[
        \Lambda(s)=\pi^{-\frac{s}{2}}\Gamma\left(\frac{s}{2}\right)\zeta(s)
    \]
    decreases very quickly in vertical strips, due to the presence of the gamma factor.
    
    In other words, the (interesting) growth behavior of the L-function in vertical strips is actually not a well-defined feature of the L-function when it is defined via the GCD procedure. However, with the new formalism, a statement of the form ``the L-function grows slowly in vertical strips'' makes sense without modifying the definition of the L-function. We will discuss this further in Appendix~\ref{app:non_canonical_triv}.
\end{remark}

\begin{remark}
    It is also possible to redo the above definitions based on different variations of the ring $S(\AA^\times)$, and therefore $\cS$. If, for instance, one chooses to loosen the smoothness requirements, this would result in Mellin transforms that are required to decrease in vertical strips, but not very quickly. That is, instead of rapidly decreasing in vertical strips, our functions would be required to decrease more slowly (say, at some fixed polynomial rate).
    
    In turn, this would make the roughening of $\cS$ smaller. The implication of that would be that the L-function's growth in vertical strips is more strictly controlled. This would make questions of Lindel\"of type well-defined.
\end{remark}

\begin{remark} \label{remark:almost_smooth}
    Another way to modify the definition uses something similar to the ``almost smooth'' functions of Appendix~A of \cite{almost_smooth_padic}. These functions are a natural modification of the notion of a smooth function on the $p$-adic part of $\AA^\times$. In brief: instead of taking locally constant functions, one takes functions whose Fourier coefficients are rapidly decreasing. Using a definition based on almost smooth functions, rather than smooth functions, should give a notion of L-function with well defined growth properties not only in the complex $\pm i\infty$ direction, but also in the conductor aspect. This generalization will be explored elsewhere.
\end{remark}

Let us return to the proof of Theorem~\ref{thm:S_prime_L_isom}.
\begin{proof}[Proof of Theorem~\ref{thm:S_prime_L_isom}]
    We have a morphism
    \[\xymatrix{
        \cS'\otimes_{\cS}\cL \ar[r] & \cS',
    }\]
    and would like to construct an inverse. Such an inverse can always be induced from a map of bornological $\cS$-modules,
    \[\xymatrix{
        \cS \ar@{-->}[r] & \cS'\otimes_{\cS}\cL.
    }\]
    
    Our strategy for doing so will work as follows. We will construct maps
    \begin{equation} \label{eq:local_canonical_trivs}\xymatrix{
        S(F_v^\times) \ar[r] & \cS'\otimes_{S(F_v^\times)} S(F_v)
    }\end{equation}
    for all places $v$, corresponding to a generator $t_v\co S(F_v^\times)\ra S(F_v)$ tensored with its inverse. Since the product defining an L-function absolutely converges in the right-half-plane of $\Chars(F)$ corresponding to $\cS'$, these maps will multiply together to yield the desired map
    \[\xymatrix{
        \cS \ar[r] & \cS'\otimes_{\cS}\cL.
    }\]
    
    Indeed, we define $\cS_v=S(F_v^\times)$ and $\cL_v=S(F_v)$. We also define a local variant of the ring $\cS'$, as follows. Set the ring extension $\cS_v'$ of $\cS_v$ to be the G\r{a}rding space of the space of functions $f$ on $F_v^\times$ such that there is a bound:
    \[
        \int_{F_v^\times}\abs{g}_v^{1+\varepsilon}\max\{1,\abs{g}_v\}^\sigma\cdot\abs{f(g)}\dtimes{g}<\infty
    \]
    for all $\sigma<\infty$, $\varepsilon>0$. We also require that $f$ be supported on a compact subset of $F_v$ when $v$ is finite.
    
    That is, when $v$ is finite, then $\cS_v'$ is the space of locally constant functions on $F_v^\times$, supported on a compact subset of $F_v$, which do not increase much faster than $\abs{g}^{-1}$ near $0$. When $v$ is infinite, then $\cS_v'$ is the space of smooth functions on $F_v^\times$ which are rapidly decreasing near $\infty$ and do not increase much faster than $\abs{g}^{-1}$ near $0$.
    
    The reader should note that in the case where $F_v$ is non-Archimedean, the term ``G\r{a}rding space'' simply refers to the subspace of functions on which the action of $F_v^\times$ is locally constant. This is always the case when $F$ is a function field.
    
    Now, our claim is that the natural inclusion $\cS_v\hookrightarrow\cL_v$ induces an isomorphism
    \[\xymatrix{
        \cS_v' \ar[r]^-\sim & \cS'_v\otimes_{\cS_v}\cL_v.
    }\]
    Indeed, the quotient $\cL_v/\cS_v$ is killed by any inverse of the local L-function at $v$, which is invertible in $\cS_v'$.
    
    So, composing the above with the map $\cS_v\ra\cS_v'$ gives the desired local maps of~\eqref{eq:local_canonical_trivs}. Thus, it remains to show the convergence property we are after. Specifically, we claim that for all finite places $v$, the image of the distinguished vector $\one_{\O_v^\times}$ is the element
    \[
        (\one_{\O_v^\times}-\one_{\pi_v\O_v^\times})\otimes\one_{\O_v}\in\cS_v'\otimes_{\cS_v}\cL_v,
    \]
    where $\pi_v$ is a uniformizer. Thus, it is enough to check that the product
    \[
        \prod_v (\one_{\O_v^\times}-\one_{\pi_v\O_v^\times})
    \]
    converges in $\cS'$, which indeed holds.
\end{proof}

\section{Functional Equation} \label{sect:functional_equation}

Our goal for this section is to provide a notion of functional equation and analytic continuation that is compatible with our formalism. In terms of the analogy of Remark~\ref{remark:L_is_divisor}, we are showing that the ``divisor'' corresponding to $\cL$ is symmetric.

To do this, we will need to examine the operation of Fourier transform on $\cL$ and its interaction with the embedding of $\cL$ in the space of functions on $\AA^\times/F^\times$. It will turn out that the compatibility (or lack thereof) of this embedding with the Fourier transform is deeply related to the poles of the L-function. We will thoroughly delve into this issue in the following two subsections.

For the moment, let us give some notation, and provide some bottom lines. We begin by defining a ring involution $\iota\co\cS\ra\cS$ by:
\[
    \iota(f)(g)=\abs{g}^{-1}f(g^{-1}).
\]
For an $\cS$-module $M$, we will denote by $\prescript{\iota}{}{M}$ the twist of $M$ by $\iota$.

\begin{construction} \label{const:fourier}
    We define the isomorphism
    \[
        \F\co S(\AA)\xrightarrow{\sim}\prescript{\iota}{}{S(\AA)}.
    \]
    to be given by the Fourier transform. This isomorphism depends on a choice of additive character $\psi\co\AA/F\ra\CC^\times$. By taking co-invariants, we also obtain an isomorphism:
    \[
        \F\co\cL\xrightarrow{\sim}\prescript{\iota}{}{\cL},
    \]
    which no longer depends on the choice of $\psi$.
\end{construction}

This means that we have two spaces, $\cS$ and $\cL$, each armed with an involution, $\iota$ and $\F$ respectively. These two spaces are embedded in a bigger space $\cS'$ together. It turns out that the involutions $\iota$ and $\F$ \emph{coincide} on the intersection $\cS\cap\cL$ inside $\cS'$. In particular, one can define a pushout diagram of $\Schw$-modules:
\begin{equation} \label{eq:S_ext_pushout} \xymatrix{
    \cS\cap\cL \ar[r] \ar[d] & \cL \ar[d] \\
    \cS \ar[r] & \cS+\cL=\cS_\ext,
}\end{equation}
where all of the spaces carry compatible involutions.

Let us take a moment to delve into the intuitive meaning of the intersection $\cS\cap\cL$. Under the Mellin transform, $\cS$ should consist of entire functions on $\Chars(F)$ that are rapidly decreasing in vertical strips (satisfying some extra conditions). Similarly, $\cL$ should consist of such functions multiplied by the L-function on $\Chars(F)$. In particular, $\cS\cap\cL$ should consist of all functions that have the same zeroes as the L-function, no poles, and satisfy some growth conditions. This means that we expect that the intersection $\cS\cap\cL$ is large, while the quotient $\P=\cL/\cS\cap\cL$ is small, and is related to the poles of the L-function.

So, take for granted for the moment that we know that the quotient $\P=\cL/\cS\cap\cL=\cS_\ext/\cS$ is sufficiently ``small''. Then we are able to interpret the above diagram~\eqref{eq:S_ext_pushout} as giving a functional equation and analytic continuation for L-functions: the extension $\cS_\ext$ of $\cS$ is essentially the space of meromorphic functions with a small number of prescribed poles, and the map $\cL\ra\cS_\ext$ which sends a test function to its corresponding zeta integral respects the involutions on both sides.

Hence, our problem is reduced to proving some smallness bounds on $\P$. Na\"ively, this sounds hard. The quotient $\cS/\cS\cap\cL$ contains all information about the zeroes of the L-function, and so an explanation is needed as to why the seemingly similar quotient $\P=\cL/\cS\cap\cL$ is so much easier to characterize.

As it turns out, we are able to provide a surprisingly \emph{conceptual} reason for why the poles of the L-function are easier to study. Specifically, we embed $\cL$ in a space $\widetilde{\cS}$, which is somewhat larger than $\cS'$, but carries its own involution. We then show that the difference between the involutions of $\cL$ and $\widetilde{\cS}$ precisely captures the polar divisor $\P$. This allows us to isolate the desired quotient for study in a canonical way.

The structure of the rest of this section is as follows. In Subsection~\ref{subsect:analytic_cont_func_eq}, we will formalize the above explanation about the various involutions involved. We will introduce a slightly different definition for $\P$, although it will turn out to be equivalent. In Subsection~\ref{subsect:polar_div}, we will study the properties of the polar divisor $\P$, and prove that it is the same as $\cL/\cS\cap\cL$.

\subsection{Analytic Continuation and the Polar Divisor} \label{subsect:analytic_cont_func_eq}

In order to be able to state a re-interpretation of the functional equation, we first need to change the target space for the map $\cL\ra\cS'$. The reason for that is that the space $\cS'$ is just too small; it is not symmetric under $\iota$. Instead, we define a new target space $\widetilde{\cS}$. It will have the advantage of being symmetric under $\iota$, but it will no longer be a ring. This will allow us to compare $\iota$ with $\F$.

\begin{construction} \label{const:schw_prime_prime}
    Define the $\cS$-module $\widetilde{\cS}$ to be the G\r{a}rding space (i.e., the space of the smooth vectors) of the space of functions $f$ on $\AA^\times/F^\times$ that are of moderate growth, with its natural bornology.
\end{construction}
\begin{remark}
    The space $\widetilde{\cS}$ is an adelic version of the space $A_\text{umg}(\Gamma\backslash G)$ of functions of uniform moderate growth (c.f. \cite{schwartz_of_aut_quotient}).
\end{remark}
\begin{remark}
    The bornological space $\widetilde{\cS}$ is complete, and by Theorem~\ref{thm:adelic_garding_is_smooth} it is also smooth over $\cS$.
    
    In fact, it is precisely the smoothening of the dual of $\cS$. I.e, $\widetilde{\cS}$ is the \emph{contragradient} of $\cS$.
\end{remark}
\begin{remark}
    In the case where $F$ is taken to be a function field, then no bornologies are necessary. One can directly define $\widetilde{\cS}$ to be the contragradient of $\cS$, in the sense of being the space of locally constant functions on $\AA^\times/F^\times$ with no conditions on growth.
\end{remark}

\begin{remark}
    The isomorphism $\iota\co\cS\xrightarrow{\sim}\prescript{\iota}{}{\cS}$ extends to an isomorphism
    \[
        \iota\co\widetilde{\cS}\xrightarrow{\sim}\prescript{\iota}{}{\widetilde{\cS}}
    \]
    via the same formula
    \[
        \iota(f)(g)=\abs{g}^{-1}f(g^{-1}).
    \]
\end{remark}

\begin{construction}
    We define a map $\cL\ra\widetilde{\cS}$ via the composition
    \[
        \cL\ra\cS'\ra\widetilde{\cS}.
    \]
\end{construction}

Given this embedding, one may naturally ask about the relation between the two involutions $\F\co\cL\ra\prescript{\iota}{}{\cL}$ and $\iota\co\widetilde{\cS}\ra\prescript{\iota}{}{\widetilde{\cS}}$. As it turns out, they do not commute, and this lack of commutativity reflects the poles of the L-function.
\begin{definition}
    We let $\delta\co\cL\ra\widetilde{\cS}$ denote the difference between the two maps in the diagram:
    \[\xymatrix{
        \cL \ar[r]_j \ar[d]^\F & \widetilde{\cS} \ar[d]^\iota \ar@{=>}[dl]^{\iota\circ\delta} \\
        \prescript{\iota}{}{\cL} \ar[r]_{\iota(j)} & \prescript{\iota}{}{\widetilde{\cS}}.
    }\]
    
    In other words, we have $\delta=j-\iota\circ\iota(j)\circ\F$, where $j\co\cL\ra\widetilde{\cS}$ is the embedding.
\end{definition}

Our claim is that the map $\delta\co\cL\ra\widetilde{\cS}$ precisely presents the \emph{poles} of our L-function. In the divisor interpretation, this means that it presents the positive part of the formal difference $[\cL]-[\cS]$. In other words, let us define:
\begin{definition}
    Define the bornological $\cS$-module $\P$ to be the co-kernel
    \[\xymatrix{
        0 \ar[r] & \ker{\delta} \ar[r] & \cL \ar[r] & \P \ar[r] & 0.
    }\]
    We refer to $\P$ as the \emph{polar divisor}.
\end{definition}
This coincides with the interpretation as the positive part of the difference $[\cL]-[\cS]$ via:
\begin{proposition} \label{prop:kernel_of_delta}
    The kernel $\ker{\delta}$ is equal to the intersection $\cS\cap\cL$ in $\widetilde{\cS}$.
\end{proposition}
\begin{remark} \label{remark:kernel_of_delta}
    The inclusion $\ker{\delta}\subseteq\cS\cap\cL$ is easy to see. Indeed, note that:
    \[
        \ker{\delta}\subseteq{\cS'}\cap\iota({\cS'}).
    \]
    However, it is clear from the definition that
    \[
        {\cS'}\cap\iota({\cS'})\subseteq\cS.
    \]
\end{remark}

We will postpone the proof of the rest of Proposition~\ref{prop:kernel_of_delta} to the next subsection.

Before we discuss the specific properties of $\P$, let us discuss its relation to the analytic continuation property and the functional equation. Informally, the idea is as follows. We are looking for a space of meromorphic functions (under some Paley-Wiener correspondence) which is simultaneously big enough to contain $\cL$, and is symmetric with respect to $\iota$. Given such a space, we would be able to ask if its involution $\iota$ is compatible with the involution $\F$ of $\cL$. That would be our functional equation.

Using the space $\widetilde{\cS}$ above fails on two counts: its involution is not compatible with $\F$, and it is too ``big'' to be thought of as a space of meromorphic functions. Therefore, we can ask for some minimal space $\cS_\ext$ satisfying our requirements.

It turns out that we can safely choose $\cS_\ext$ to be the pushout
\[\xymatrix{
    \ker{\delta} \ar[d] \ar[r] & \cL \ar@{-->}[d] \\
    \cS \ar@{-->}[r] & \cS_\ext.
}\]
This space automatically contains both $\cS$ and $\cL$, and has an involution
\[
    \cS_\ext\xrightarrow{\sim}\prescript{\iota}{}{\cS_\ext}
\]
compatible with both $\iota\co\cS\ra\prescript{\iota}{}{\cS}$ and $\F\co\cL\ra\prescript{\iota}{}{\cL}$. Moreover, clearly $\cS_\ext$ is an extension of $\cS$ by $\P$. Thus, if we had a strong enough grip on the support of $\P$, we would immediately be able to interpret $\cS_\ext$ as a space of meromorphic functions with prescribed poles, and see that the image of $\cL$ in $\cS_\ext$ satisfies a functional equation.

\begin{remark} \label{remark:kernel_delta_means_S_ext_in_S_prime}
    The essence of Proposition~\ref{prop:kernel_of_delta} is now that $\cS_\ext$ can be identified with the sum $\cS+\cL$ inside $\widetilde{\cS}$. I.e., that the two inversions $\iota$, $\F$ can be extended to the sum $\cS+\cL$ in a compatible manner. Alternatively, since both $\cS$ and $\cL$ lie inside ${\cS'}$, the proposition can be interpreted as meaning that the natural map
    \[
        \cS_\ext\ra{\cS'}
    \]
    is injective.
    
    Since we are informally thinking of ${\cS'}$ as a space of holomorphic functions on some right-half-plane, Proposition~\ref{prop:kernel_of_delta} should be interpreted as meaning that the functional equation carries no ``extra'' data in its L-functions that cannot be analytically continued from the right-half-plane. That is, a-priori it could have been possible that the correct notion of L-function contained a delta distribution on the complex plane (which would be invisible to analytic continuation). The proposition excludes this possibility.
\end{remark}

\subsection{Properties of the Polar Divisor \texorpdfstring{$\P$}{P}} \label{subsect:polar_div}

Let us now start studying the properties of $\P$. Our goal is to show that its support is sufficiently small that the space $\cS_\ext$ of the previous subsection corresponds to meromorphic functions on the whole of $\CC$. Finally, we will use this to prove Proposition~\ref{prop:kernel_of_delta}.

\begin{remark} \label{remark:P_is_smooth}
    We first note that $\P$ is smooth, as it is the quotient of a smooth module by a closed subspace. This follows from Corollary~5.6 of \cite{dixmier_malliavin_for_born_arxiv}.
\end{remark}

\begin{remark}
    In fact, using Proposition~\ref{prop:formula_for_delta}, we will see that we know exactly what $\P$ is in this case. It turns out to be $2$-dimensional, and in particular the bornological $\cS$-module $\P$ is complete.
\end{remark}

\begin{remark}
    As our next step in characterizing $\P$, we note that we must have a canonical isomorphism:
    \[
        \P\xrightarrow{\sim}\prescript{\iota}{}{\P},
    \]
    induced by $\F$. That is, the polar divisor is symmetric about the critical strip. This follows because
    \[
        \iota\circ\delta\circ\F=-\delta,
    \]
    and thus $\ker{\delta}$ is preserved by $\F$.
\end{remark}

In order to make any further progress, we will need to make use of the following formula for $\delta$, which is an incarnation of the Poisson summation formula:
\begin{proposition} \label{prop:formula_for_delta}
    The function $\delta\co\cL\ra\widetilde{\cS}$ is given by the formula:
    \[
        \delta(\Psi)(g)=(\F\Psi)(0)\abs{g}^{-1}-\Psi(0).
    \]
\end{proposition}

This will allow us to know exactly where the poles $\P$ are, and to prove Proposition~\ref{prop:kernel_of_delta}.

It is now clear exactly what $\P$ is:
\begin{corollary} \label{cor:supp_of_polar}
    There is an isomorphism:
    \begin{equation*}
        \one\oplus\prescript{\iota}{}{\one}\xrightarrow{\sim} \P,
    \end{equation*}
    where $\one$ is the trivial $\AA^\times$-module, thought of as an $\cS$-module.
\end{corollary}
In particular, $\P$ is a sum of skyscraper co-sheaves on $\Chars(F)$.

Using this corollary, we can also finally prove Proposition~\ref{prop:kernel_of_delta}.
\begin{proof}[Proof of Proposition~\ref{prop:kernel_of_delta}]
    We want to show that $\ker{\delta}=\cS\cap\cL$. The inclusion of the LHS in the RHS was already shown in Remark~\ref{remark:kernel_of_delta}.
    
    The inclusion in the other direction will use Corollary~\ref{cor:supp_of_polar}. We turn $\cS\cap\cL$ into a bornological vector space via the Cartesian diagram
    \[\xymatrix{
        \cS\cap\cL \ar[r] \ar[d] & \cS \ar[d] \\
        \cL \ar[r] & \widetilde{\cS}.
    }\]
    First, we observe that the restriction $\iota\circ\delta|_{\cS\cap\cL}\co\cS\cap\cL\ra\prescript{\iota}{}{\widetilde{\cS}}$ factors through $\prescript{\iota}{}{{\cS'}}$. Indeed, it is given by the difference of maps $\iota-\F$, and the restriction lies in $\cS+\prescript{\iota}{}{\cL}\subseteq\prescript{\iota}{}{{\cS'}}$ (we are using the fact that the map ${\cS'}\ra\widetilde{\cS}$ is injective).
    
    Hence, we obtain an injective map
    \[
        \frac{\cS\cap\cL}{\ker{\delta}}\ra\prescript{\iota}{}{{\cS'}},
    \]
    and want to show that its domain is $0$. Since $\P$ is (in an informal sense) torsion by Corollary~\ref{cor:supp_of_polar}, and the domain in question is a subspace of $\P$, it is enough to show that ${\cS'}$ is torsion-free in the same sense. That is, it is enough to show that if $f\in\cS'$ satisfies
    \[
        (f(ghh')\abs{h}-f(gh'))-(f(gh)\abs{h}-f(g))=0
    \]
    for all $h,h'\in\AA^\times$, then $f(g)=0$. However, this is obviously correct.
\end{proof}

\section{Abelian Extensions} \label{sect:decomposition_under_ext}

Let $E\supseteq F$ be an Abelian extension. For the sake of simplicity, we suppose that $E$ is quadratic over $F$. Then the zeta function of $E$ factors into a product of two L-functions over $F$. In this section, we will present this statement's incarnation in our language. This turns out to be an actual refinement; the new statement contains additional information allowing one to relate zeta integrals of specific test functions on $E$ with zeta integrals of specific test functions on $F$.

We begin by introducing some notation. Denote the character on $\AA^\times_F/F^\times$ corresponding to the extension $E/F$ by $\eta=\eta_{E/F}$. This induces an automorphism
\[
    \chi\mapsto\chi\eta
\]
of $\Chars(F)$, along with an automorphism
\begin{align*}
    \eta\co\cS_F & \ra\cS_F \\
    f(g) & \mapsto\eta(g)\cdot f(g)
\end{align*}
of the ring $\cS_F$. For a $\cS_F$-module $M$, we will let $\prescript{\eta}{}{M}$ denote the twist of $M$ by $\eta$.

From the geometric point of view, the extension $E\supseteq F$ induces a canonical map
\[
    N\co\Chars(F)\ra\Chars(E),
\]
given by pre-composing a Hecke character $\chi\co\AA_F^\times/F^\times\ra\CC^\times$ with $N_{E/F}$. Moreover, we get a morphism of non-unital rings
\[
    N_!\co \cS_E\ra\cS_F,
\]
given by integrating along the multiplicative norm map
\[
    N_{E/F}\co\AA_E^\times\ra\AA_F^\times.
\]

We now claim that, informally, when $\cL_E$ is ``pulled back'' to $\Chars(F)$ along $N$, it splits into a product. This will recover the corresponding classical fact about decomposition of L-functions for quadratic extensions.
\begin{theorem} \label{thm:L_decomposes}
    There is a canonical isomorphism of bornological $\cS_F$-modules,
    \[
        \cS_F\otimes_{\cS_E}\cL_E\cong\cL_F\otimes_{\cS_F}\prescript{\eta}{}{\cL_F}.
    \]
    Moreover, the isomorphism between the two sides is also compatible with the maps into $\cS'_F$.
\end{theorem}

Before diving into the proof, let us re-interpret Theorem~\ref{thm:L_decomposes} in terms of the analogy of Remark~\ref{remark:L_is_divisor}, to get a more geometric intuition.
\begin{remark}
    Consider the canonical map
    \[
        N\co\Chars(F)\ra\Chars(E),
    \]
    given by pre-composing a Hecke character with $N_{E/F}$. This map has kernel $\{1,\eta\}$, and its image is the subgroup of $\Chars(E)$ given by characters that are invariant under the Galois group of $E$ over $F$. Let us denote this image by $\Chars(E/F)=\Chars(E)^{\Gal(E/F)}$.
    
    With this language, the informal essence of Theorem~\ref{thm:L_decomposes} is that the push-forward of the ``divisor'' $\cL_F$ to $\Chars(E/F)$ (as a divisor) identifies with the restriction of $\cL_E$ to $\Chars(E/F)$.
\end{remark}

\begin{proof}[Proof of Theorem~\ref{thm:L_decomposes}]
    We will prove that
    \[
        S(\AA_F^\times)\otimes_{S(\AA_E^\times)}S(\AA_E)\cong S(\AA_F)\otimes_{S(\AA_F^\times)}\prescript{\eta}{}{S(\AA_F)}.
    \]
    
    This will be proven by embedding both sides in the larger space $\Func(\AA_F^\times)$ of functions on $\AA_F^\times$, and showing that their images coincide. Furthermore, we need to show that the bornologies on the two sides coincide. The two embeddings will be induced by the two maps,
    \begin{equation*}\xymatrix{
        S(\AA_E) \ar[r] & \Func(\AA_F^\times), \\
        S(\AA_F)\otimes_{S(\AA_F^\times)}\prescript{\eta}{}{S(\AA_F)} \ar[r] & \Func(\AA_F^\times),
    }\end{equation*}
    given by
    \[
        \left\{h\mapsto f(h)\right\}\mapsto\left\{g\mapsto \int_{N(h)=g}f(h)\dtimes{h}\right\}
    \]
    and
    \[
        f_1\otimes f_2\mapsto f_1*\eta f_2=\left\{g\mapsto \int f_1(g'^{-1}g)f_2(g')\eta(g')\dtimes{g'}\right\}
    \]
    respectively.
    
    To check this, it is enough to check that the above claim holds place-by-place, which is a straightforward (albeit tedious) verification.
\end{proof}

\begin{remark} \label{remark:abelian_L_decomposes}
    Let $E/F$ be an Abelian extension, which is no longer necessarily quadratic. Let us state the relevant generalization of Theorem~\ref{thm:L_decomposes}.
    
    There are characters $\eta_i\co\AA_F^\times/F^\times\ra\CC^\times$ corresponding to the extension $E/F$, with $0\leq i\leq d-1$ and where the degree of $E$ over $F$ is $d$. Moreover, there are maps
    \begin{align*}
        N & \co\Chars(F)\ra\Chars(E) \\
        N_! & \co \cS_E\ra\cS_F,
    \end{align*}
    as above. The claim is that the canonical maps into $\cS_F'$ induce an isomorphism
    \[
        \cS_F\otimes_{\cS_E}\cL_E\cong\prescript{\eta_0}{}{\cL_F}\otimes_{\cS_F}\cdots\otimes_{\cS_F}\prescript{\eta_{d-1}}{}{\cL_F}.
    \]
\end{remark}

\begin{remark} \label{remark:cubic_L_decomposes}
    One can also generalize the above to extensions that are not necessarily Abelian. For example, suppose that $E/F$ is a non-Abelian cubic extension. In this case, the extension defines an irreducible generic automorphic representation $(\pi,V)$ of $\GL_2(\AA_F)$. The author believes (but has not proven) that the correct variant of Remark~\ref{remark:abelian_L_decomposes} is as follows.
    
    We define a $S(\AA_F^\times)$-module by restriction of $\pi$ to $\GL_1(\AA_F)\times \{1\}$ along the diagonal. We denote its co-invariants under the resulting action of $F^\times$ by $\cL_F(\pi)$. There is an embedding $V\subset\Func(\AA_F^\times)$ given by the Whittaker model. After taking quotients by $F^\times$, this gives a map $\cL_F(\pi)\ra\cS'_F$.
    
    Now, the author believes (but has not proven in general) that this induces an isomorphism
    \[
        \cS_F\otimes_{\cS_E}\cL_E\cong\cL_F\otimes_{\cS_F}\cL_F(\pi).
    \]
    The author finds it interesting that the object $\cS_F\otimes_{\cS_E}\cL_E$, constructed from \emph{Galois} data of the extension $E/F$, can be directly related to the underlying space $V$ of the \emph{automorphic} representation $\pi$, through its quotient $\cL_F(\pi)$.
    
    That is, it seems that the factorization of the L-function for non-Abelian field extension allows constructing a correspondence between the underlying spaces of the automorphic representations associated to the field extension, and a construction made of pure Galois data. 
\end{remark}

\begin{appendices}
\section{Generators for \texorpdfstring{$\cL$}{L}} \label{app:non_canonical_triv}

The goal of this section is to explicitly show that the $\cS$-module $\cL$ defined above happens to be free of rank one. This will be the main result of this section, Theorem~\ref{thm:L_is_triv}. The choice of generator for $\cL$ is analogous to the process of picking a standard L-factor at every place in the GCD description of L-functions. In particular, the generator itself is not well-defined, and therefore the constructions below will be somewhat ad hoc.

For the non-Archimedean places, we will see that the standard L-factor used suits our purposes just fine. That is, it serves as a generator for an appropriate module. See Claim~\ref{claim:locally_trivial_at_Qp}.

For Archimedean places, this is no longer true. That is, the standard choice of L-factor at the Archimedean places turns out not to be a generator for the appropriate module. The reason for this failure is that the standard L-factor decreases too quickly in vertical strips. Instead, we will merely show the existence of a modification for this L-factor which \emph{does} have the right growth properties to be a generator. This is the content of Claim~\ref{claim:locally_trivial_at_R}. See also Remark~\ref{remark:growth_of_L_in_vertical_strips}.

\begin{theorem} \label{thm:L_is_triv}
    The $\cS$-module $\cL$ is isomorphic to $\cS$.
\end{theorem}
This will follow from:
\begin{claim} \label{claim:A_is_A_times}
    The $S(\AA^\times)$-module $S(\AA)$ is isomorphic to $S(\AA^\times)$.
\end{claim}
We will prove this place-by-place.

\begin{claim} \label{claim:locally_trivial_at_Qp}
    Let $F$ be a non-Archimedean local field. Then the $S(F^\times)$-module $S(F)$ is isomorphic to $S(F^\times)$. Moreover, the isomorphism can be chosen such that it sends $\one_{\O^\times}$ to $\one_\O$.
\end{claim}
This is also proven as Item~(2) of Lemma~4.18 of \cite{zeta_rep}.
\begin{proof}
    We consider the morphism
    \[
        S(F^\times)\ra S(F)
    \]
    given by convolution with the distribution
    \[
        f(g)=\begin{cases}
            0 & \abs{g}>1 \\
            \delta_1(g) & \abs{g}=1 \\
            1 & \abs{g}<1,
        \end{cases}
    \]
    where $\delta_1(g)$ is the delta distribution at $g=1$. It is a direct verification to check that this map satisfies the required properties.
\end{proof}

\begin{claim} \label{claim:locally_trivial_at_R}
    Let $F=\RR$. Then the $S(\RR^\times)$-module $S(\RR)$ is isomorphic to $S(\RR^\times)$.
\end{claim}
\begin{proof}
    One way to think about this kind of isomorphism is via the Mellin transform. A map
    \[
        S(\RR^\times)\ra S(\RR)
    \]
    should correspond to pointwise multiplication by some function in the Mellin picture. In order to have the right image, this function (which is, essentially, the local L-function) needs to have the right poles and zeroes, as well as some growth properties in vertical strips.
    
    There are some explicit maps which are almost isomorphisms of $S(\RR)$ with $S(\RR^\times)$. The Mellin transforms of functions such as $e^{-\pi y^2}$ and $e^{iy^2}$ have the correct set of poles to give the right ``divisor'', but decrease too fast as $s\ra-i\infty$. The strategy of our proof will be to choose a function with the right poles and zeroes, and then ``fix'' its vertical growth.
    
    Let us address the proof itself. Note that it is enough to prove the claim separately for even and odd functions on $\RR$. We will define generalized functions $\phi_\pm\co\RR\ra\CC$, one for each parity, which are rapidly decreasing at $\infty$, and such that convolution with $\phi_++\phi_-$ defines the sought-after isomorphism
    \[
        S(\RR^\times)\xrightarrow{\sim}S(\RR).
    \]
    We will explicitly describe only $\phi_+$. The odd variant can be given by $\phi_-(y)=y\phi_+(y)$. We will describe $\phi_+(y)$ via its Mellin transform.
    
    The idea is this. The usual L-function at $\infty$,
    \[
        f(s)=\pi^{-\frac{s}{2}}\Gamma\left(\frac{s}{2}\right),
    \]
    has the right set of zeroes, but it decreases very quickly in vertical strips. This would prevent it from defining an isomorphism. Specifically, its absolute value behaves as
    \[
        \abs{\left(\frac{s}{2e\pi}\right)^{s/2}\sqrt{\pi s}}\sim\abs{t}^{\frac{\sigma+1}{2}}e^{-\frac{\pi}{4}\abs{t}}
    \]
    via the Stirling formula, when $s=\sigma+it$ and $t\ra\pm\infty$. However, applying Lemma~\ref{lemma:correct_vertical} below to $f(s)$ yields an entire function $g(s)$, which has no zeroes, has a simple pole at every non-positive even integer, and such that $g(s)$ is bounded from above and below in vertical strips.
    
    We now claim that multiplying the Mellin transform by $g(s)$ gives an isomorphism. Specifically, let
    \[
        \alpha\co S(\RR^\times)_+\ra S(\RR)_+
    \]
    be the map on even functions sending a function whose Mellin transform is $h(s)$ to the function whose Mellin transform is $h(s)g(s)$. We wish to show that this map is bijective.
    
    Indeed, the map $\alpha$ is clearly injective. It remains to show that it is a surjection. Let $u(s)$ be the Mellin transform of an even function $v(y)$ in $S(\RR)_+$. We wish to show that $u(s)/g(s)$ is entire and is rapidly decreasing in vertical strips. That it is entire is clear, since $u(s)$ has at most simple poles at negative even integers. The rapid decrease in vertical strips follows by induction, using the fact that
    \[
        \frac{v(y)-v(0)e^{-\pi y^2}}{y}
    \]
    removes the pole at $0$ and shifts $u(s)$ by $1$.
\end{proof}

In the course of the proof, we have used the following lemma to correct the behaviour of a function in vertical strips.
\begin{lemma} \label{lemma:correct_vertical}
    Let $f(s)$ be a meromorphic function on $\CC$, which has no zeroes or poles outside the horizontal half-strip $\{\sigma+it\suchthat\text{$\sigma<1$ and $\abs{t}<1$}\}$. Then there exists a meromorphic function $g(s)$ such that $f(s)/g(s)$ has no zeroes or poles, and $g(s)$ is uniformly bounded from above and below outside the half-strip.
\end{lemma}
\begin{proof}
    This is a direct consequence of Arakelian's approximation theorem (see \cite{entire_function_approx}). The theorem allows us to create an entire function $\phi(s)$ satisfying that $\abs{\phi(s)-\log(f(s))}$ is uniformly bounded outside the half-strip. The desired function is now
    \[
        g(s)=e^{-\phi(s)}f(s).
    \]
\end{proof}

\begin{claim} \label{claim:locally_trivial_at_C}
    Let $F=\CC$. Then the $S(\CC^\times)$-module $S(\CC)$ is isomorphic to $S(\CC^\times)$.
\end{claim}

Recall that the notation $\abs{\cdot}_\CC=\abs{N_{\CC/\RR}(\cdot)}$ denotes the absolute value of the norm, and thus differs from the usual absolute value by a power of $2$.

\begin{remark} \label{remark:paley_wiener_at_C}
    In order to prove Claim~\ref{claim:locally_trivial_at_C}, we will need to give a Paley-Wiener style description for $S(\CC^\times)$. Indeed, we have
    \[
        S(\CC^\times)=S\!\left(\RR^\times_{>0}\times\RR/{2\pi i\ZZ}\right)=S(\RR^\times_{>0})\,\hat{\otimes}\,S\!\left(\RR/{2\pi i\ZZ}\right)
    \]
    Thus, the Mellin transform
    \[
        \hat{f}_n(s)=\int f(z)\left(\frac{z}{\abs{z}}\right)^{n}\abs{z}_\CC^s\dtimes{z}
    \]
    of a function $f\in S(\CC^\times)$ is entire in $s$, and satisfies that its semi-norms
    \[
        \norm{\hat{f}_n(s)}_{\sigma,m}=\sup_{\stackrel{t\in\RR}{n\in\ZZ}}(1+\abs{t}^m)(1+\abs{n}^m)\abs{\hat{f}_n(\sigma+it)}
    \]
    are bounded for all $\sigma\in\RR$ and $m\geq 0$. This description exactly characterizes the image of $S(\CC^\times)$ under the Mellin transform.
    
    In addition, we will make use of the following description of the Mellin transform of $S(\CC)$. A sequence $\{{f}_n(s)\}_{n\in\ZZ}$ of meromorphic functions lies in the image of the Mellin transform of $S(\CC)$ if and only if:
    \begin{enumerate}
        \item Each function ${f}_n$ has at most simple poles, and they are located inside the set $-\frac{\abs{n}}{2}-\ZZ_{\geq 0}$.
        \item The semi-norms $\norm{{f}_n(s)}_{\sigma,m}$ are bounded for all $\sigma\in \frac{1}{4}+\frac{1}{2}\ZZ$ and $m\geq 0$.
    \end{enumerate}
    This description can be proven via standards methods, using the fact that the Mellin transforms of the functions $z^m e^{-\pi\abs{z}^2}$ satisfy the above requirements.
\end{remark}

\begin{proof}[Proof of Claim~\ref{claim:locally_trivial_at_C}]
    In a similar manner to the proof of Claim~\ref{claim:locally_trivial_at_R}, it is sufficient to supply a sequence of functions $\{{g}_n(s)\}_{n\in\ZZ}$ such that:
    \begin{enumerate}
        \item Each function $g_n(s)$ has no zeroes, and has a simple pole at $-\frac{\abs{n}}{2}-m$ for all integer $m\geq 0$.
        \item The functions $g_n(s),g_n(s)^{-1}$ satisfy some moderate growth condition. We will choose $\{{g}_n(s)\}_{n\in\ZZ}$ such that:
        \[
            \sup_{\stackrel{t\in\RR}{n\in\ZZ}}\abs{{g}_n(\sigma+it)}<\infty, \qquad\sup_{\stackrel{t\in\RR}{n\in\ZZ}}\abs{{g}_n(\sigma+it)}^{-1}<\infty
        \]
        for all $\sigma\in \frac{1}{4}+\frac{1}{2}\ZZ$.
    \end{enumerate}
    
    Our choice is simply
    \[
        g_n(s)=g(\abs{n}+2s),
    \]
    where $g(s)$ is the same function as in the proof of Claim~\ref{claim:locally_trivial_at_R}.
\end{proof}

This completes the proof of Theorem~\ref{thm:L_is_triv}.

\end{appendices}

\Urlmuskip=0mu plus 1mu\relax
\bibliographystyle{alphaurl}
\bibliography{L_func}

\end{document}